\documentclass{amsart}
\usepackage{amsmath}
\usepackage{amssymb}
\usepackage{amscd}
\usepackage{amsrefs}
\usepackage{latexsym}
\usepackage{graphicx} 
\usepackage{pinlabel}

\theoremstyle{plain}
\newtheorem{theorem}{Theorem}[section]
\newtheorem{lemma}[theorem]{Lemma}
\newtheorem{cor}[theorem]{Corollary}
\newtheorem{proposition}[theorem]{Proposition}
\newtheorem*{question}{Question}

\newtheorem*{themain}{Theorem \ref{t:main}}
\newtheorem*{ther0concsurg}{Theorem \ref{t:0concsurg}}
\newtheorem*{theriso1}{Theorem \ref{t:iso1}}
\newtheorem*{theriso2}{Theorems \ref{t:infiso} and \ref{t:iso3}}
\newtheorem*{ther0iso}{Theorem \ref{t:0conciso}}

\theoremstyle{definition}

\theoremstyle{remark}
\newtheorem{remark}{Remark}

\def \- {\!\smallsetminus\!}

\def \h {\widetilde{h}}

\begin{document}

\baselineskip.5cm
\title {Surfaces in 4-manifolds: Concordance, Isotopy, and Surgery}
\author[Nathan S. Sunukjian]{Nathan S. Sunukjian}
\curraddr{Max Planck Institute for Mathematics, Bonn, Germany}
\address{Department of Mathematics, Stony Brook University \newline
\hspace*{.375in}Stony Brook, New York 11794}
\email{\rm{nsunukjian@math.sunysb.edu}}
\begin{abstract}
In this paper we will show that two surfaces of the same genus and homology class in a simply connected 4-manifold are concordant. We will show they are often topologically isotopic when their complements have cyclic fundamental group. Finally, we will show that if they are 0-concordant, then surgery on one is equivalent to surgery on the other. 

 \end{abstract}
\maketitle

\section{introduction}

Understanding the smoothly embedded surfaces in a 4-manifold is intimately related to understanding the smooth structures on a 4-manifold. This is exhibited in, for example, the following two situations. First, any $h \in H_2(X)$ can represent by a smoothly embedded surface in $X$, but the minimum genus of such a surface often depends on the particular smooth structure on $X$ \cite{KM}. Second, understanding surfaces is important, because all smooth structures can be constructed via surgery along smooth structures. Specifically, Wall has shown in \cite{W} that every smooth structure on $X$ can be constructed via a series of spherical surgeries in $X\# nS^2\times S^2$. Similarly, Iwase has shown that every smooth structure on $X$ can be obtained via a series of log transforms on tori, and Inanc Baykur and the author have given an independent proof in much the same spirit as Wall (see \cite{Iw} and \cite{BSlog}).

Unfortunately, understanding all of the embedded surfaces in a 4-manifold is an intractable problem. Not only is it already hard to classify embedded surfaces up to topological isotopy, but even within a given topological isotopy class there can be many distinct smooth embeddings. Fintushel and Stern have exhibited infinite families of surfaces which are topologically isotopic but smoothly distinct (\cite{FSsurf}, see Ruberman-Kim \cite{RK1} for variations of this as well). However, to construct every smooth structure on a 4-manifold it \emph{not} is necessary to consider surgery on every possible element of this great myriad of surfaces. Very often surgery on one surface results in the same smooth structure as surgery on a different surface.
 
The motivation for this paper is to partially make sense of this phenomenon by studying surfaces up to the equivalence of concordance. Two surfaces $\Sigma_0$ and $\Sigma_1$ embedded in a 4-manifold $X$ are called concordant if there is an embedded 3-manifold $\Sigma \times I \longrightarrow X\times I$ such that $\partial (\Sigma\times I) = \Sigma_0 \times \{0\} \cup \Sigma_1 \times \{1\}$.

We will prove the following.
\begin{themain}
Let X be a simply connected manifold. Then two smoothly (resp. locally flat) embedded surfaces  $\Sigma_0$ and $\Sigma_1$ are smoothly (resp. locally flat) concordant if and only if they have the same genus and are in the same homology class.
\end{themain}

We will apply our study of concordance to two questions, one smooth, and the other topological.

The smooth question we investigate is the following: Consider two concordant surfaces. Is the set of manifolds obtained by surgery on one of the surfaces the same as the set of manifolds obtained by surgery on the other? We define surgery along a surface as the process of removing a neighborhood of the surface and replacing it with something else.\footnote{To be even more precise, define surgery along $\Sigma_0 \subset X$ as $X'= (X - \nu \Sigma_0) \cup_f M$ for some 4-manifold $M$ and a diffeomorphism $f:\partial M \longrightarrow \partial \nu \Sigma_0$.} In general, surgeries on different surfaces result in different manifolds, but in many cases it depends on the particular type of surgery. Concordance would seem to be a natural context in which to investigate this question: By doing (surgery$\times I$) along a concordance, we get a cobordism between surgery on one surface and surgery on another. For certain kinds of surgery, this is actually an h-cobordism. The failure of the h-cobordism prevents our us from concluding that surgeries on concordant surfaces must give diffeomorphic manifolds.

Nevertheless, for certain types of concordance, we can say more. We call two genus-$g$ surfaces \emph{0-concordant} if there is a concordance $Y \subset X\times I$ such that the regular level sets $Y\cap (X\times \{ t\})$ consist of a surface of genus-$g$ and a disjoint collection of $S^2$'s. We will prove that the set of manifolds that arise via surgeries on 0-concordant surfaces are the same.

\begin{ther0concsurg}

If (surgery$\times I$) along a 0-concordance in $X\times I$ results in a cobordism between simply connected 4-manifolds, then the cobordism is trivial, and the manifolds are diffeomorphic.

\end{ther0concsurg}

The notion of 0-concordance was defined for spheres in Paul Melvin's thesis, where he shows that 0-concordant 2-knots in $S^4$ have diffeomorphic Gluck twists. Our theorem is a generalization of this, and we will further explain how it can be applied in Section \ref{s:applications}.

The topological question we address is the following: When does the homology class and genus specify a surface up to isotopy? Based on the examples in \cite{FSsurf}, one could conjecture that such simple criteria are never sufficient in the smooth category. Our focus here will be to understand when two surfaces are \emph{topologically} isotopic.

\begin{theriso1}
 If $\Sigma_0$ and $\Sigma_1$ are locally flat embedded surfaces surfaces of the same genus and homology class such that $\pi_1(X - \Sigma_i)= 0$, then the two surfaces are topologically isotopic.
\end{theriso1} 

This has been known in the case that the $\Sigma_i$ are spheres for some time. See for example \cite{LWsurf}.  The condition that the complement of the surface must be simply connected is somewhat restrictive. It implies that the homology class of the surfaces must be primitive. If the homology class is not primitive, then the simplest possible fundamental group is a cyclic group (whose order is the degree of the homology class). In this case we obtain: 

\begin{theriso2}
Let $\Sigma_0$ and $\Sigma_1$ be locally flat embedded surfaces of the same genus and homology class in a simply connected 4-manifold $X$. The surfaces are topologically isotopic when 
\begin{itemize}
\item $\pi_1(X - \Sigma_i) = \mathbb{Z}$ and $b_2 \geq |\sigma| + 6$ or,
\item $\pi_1(X - \Sigma_i) = \mathbb{Z}_n$, $b_2 > |\sigma| + 2$, and the genus of $\Sigma_0$ is strictly greater than the minimal genus for such a surface in its homology class (minimal, that is, among surfaces with $\pi_1(X - \Sigma) = \mathbb{Z}_n$.)

\end{itemize}

\end{theriso2}

We prove this in Section \ref{s:isotopy}, where we will also state a formula from \cite{LWsurf} which specifies the minimal genus.

\begin{ther0iso} 
Let $\Sigma_0$ and $\Sigma_1$ be smoothly embedded, 0-concordant surfaces in a 4-manifold $X$. If $\pi_1(X - \Sigma_i)= 0$, then $(X,\Sigma_0)$ is diffeomorphic to $(X,\Sigma_1)$.
\end{ther0iso}

We will also prove counterparts for all of the above theorems in the non-simply connected case, but the conditions will be much greater, so we will defer them until later.

We will prove Theorem \ref{t:main}, on which all of our subsequent theorems are based, in Section \ref{s:concordance}. Before that, this paper contains two  technical sections concerning surgery on codimension-2 embeddings. Specifically, Section \ref{s:surgery1} will be particularly concerned with keeping track of fundamental groups, and Section \ref{s:surgery2} will focus on the case of ambient Dehn surgery on 3-manifolds in a 5-manifold (with special emphasis on the importance of spin structures). We will construct concordances through a series of these kinds of ambient surgeries. With this outline in mind, it is possible to skip straight to Section \ref{s:concordance}, the core of the paper, and refer back to the technical details as necessary.

\section{Examples}\label{s:applications}

In some sense the most basic class of surfaces in 4-manifolds are the 2-knots in $S^4$. We begin this section by recalling what is known about concordance, isotopy, and surgery involving 2-knots, and follow that up by explaining how the work of this paper generalizes this to surfaces in other simply connected 4-manifolds.

\begin{itemize}
\item Concordance: Unlike the case for classical knots in $S^3$, all 2-knots were shown by Kervaire to be concordant to the unknot, \cite{Ker}. We give a short proof of this in Section \ref{ss:KerSlice}.  
\item Isotopy: Freedman has shown that if $\pi_1(S^4-K) = 1$, then $K$ is topologically the unknot. In general, 2 knots are determined neither by the fundamental group, (\cite{suc}), nor by $\pi_2$ thought of as a left $\mathbb{Z}[\pi_1]$ module, (\cite{Plotsuc}). It is unknown whether the homotopy type of $S^4-K$ determines $S^4-K$, even for ribbon knots.
\item Surgery: There are two natural surgeries on a 2-knot. The first is spherical surgery, where we replace the neighborhood of the knot, $S^2\times D^2$ with $S^1\times D^3$. Concordant spherical surgeries are often distinct: Surgery on knots with distinct fundamental group, for example. The second is Gluck twist, which we will concentrate on.  

\end{itemize}

\subsection{Gluck twists}
Let $S^2= K \hookrightarrow X^4$ be a knotted embedding with trivial normal bundle. The Gluck twist along this $S^2$ is defined as $X_K = X - \nu S^2 \cup_f S^2\times D^2$.\footnote{The map $f:S^1 \times S^2\longrightarrow S^1\times S^2$ is defined by $(\theta,x) \mapsto (\theta,h_{\theta}x)$ where $\h_{\theta}$ is a rotation of $S^2$ through its poles of angle $\theta$, but this is unimportant for the purposes of this paper. This is the only possible choice of $f$ that can possibly yield an $X_K$ smoothly distinct from $X$.} One can show by hand that a Gluck twist on the unknot gives back $S^4$. It is unknown whether a Gluck twist on a knot in $S^4$ ever gives something other than $S^4$, and the situation is understood only for certain classes of knots:

If $K$ is a 2-knot in $S^4$, then a Gluck twist on $K$ gives back $S^4$ when $K$ is:
\begin{enumerate}

\item a ribbon knot
\item a spun knot (\cite{Gl})
\item a twist-spun knot (\cite{G},\cite{pao})
\item a band sum of a link whose components are either ribbon or twist spun \cite{GlBook}.
\item 0-concordant to the unknot \cite{M}
\end{enumerate}

Our Theorem \ref{t:0concsurg} gives a new proof of (5), which implies (1), (2) and, (4), but not (3)\footnote{For completeness we will outline a proof that Gluck twists on twist-spun knots are trivial: The resulting manifold is a homotopy 4-sphere admitting an effective $S^1$ action. By a result of Pao \cite{pao}, any such manifold must be diffeomorphic to $S^4$.}. One might ask how big each of these classes are. All spun knots are ribbon knots, while  that not all knots are ribbon knots (see e.g. \cite{Rub},\cite{Yaj1}, \cite {Yaj2}, \cite{Coc}). In fact, these papers demonstrate that there are twist-spun knots that are not ribbon knot. Cochran's twist-spun knots in \cite{Coc} are, however, 0-concordant to the unknot.\footnote{There are a few other knots that don't obviously fall in to the categories above, that nevertheless yield back $S^4$ after Gluck twisting. Examples include \cite{GCS1},\cite{GCS2},\cite{GCS3},\cite{AkCS},\cite{NS}. Akbulut and Yasui have recently studied Gluck twists on homologically essential spheres in \cite{AY}, and shown that in many cases they do not change the smooth structure. On the other hand, Akbulut has examples of Gluck twists in non-orientable manifold that change the manifold without changing the simple homotopy type, (\cite{Afake}).} Melvin asks the following:

\begin{question}
Are all 2-knots in $S^4$ 0-concordant to the unknot?
\end{question} 
An affirmative answer to this question would imply that the Gluck twist is a useless operation for constructing exotic $S^4$'s.

\subsection{Other Surgeries}

In this paper ``surgery'' will mean the process of removing the neighborhood of a surface (or configuration of surfaces), and replacing it with something else. In addition to spherical modifications and Gluck twists, we know of four more types of surface surgeries that have been useful in the study of 4-manifolds.
\begin{enumerate}
\item Logarithmic transformation: Cut out the neighborhood of a torus $T$ that has $[T]^2=0$ and re-glue it in a different way. 
\item Knot surgery: Cut out the neighborhood of a torus $T$ that has $[T]^2=0$, and replace it with $S^1\times (S^3 -  \nu K)$ for some knot $K\subset S^3$ (in many cases the particular gluing will not matter). 

\item Rational blowdown: Replace certain plumbings of 2-spheres with a rational 4-ball. 

\item Fiber Sums: Take two manifolds $X_1$ and $X_2$ possessing Lefschetz fibrations of the same genus, and glue them together along complements of the fiber, $X = X_1 - \nu(f_1) \cup X_2 - \nu(f_2)$.

\end{enumerate}

In general these surgeries change the homotopy type of the manifold, although in certain circumstances, for example if the complement of $T$ is simply connected, they will not chance the homeomorphisms type, but only change the diffeomorphism type. Let us give a few examples of concordant surfaces, and spell out explicitly what Theorem \ref{t:0concsurg} is telling us with respect to some of the surgeries above. One easy way to construct a new embedded surface from and old one is by connect summing with a 2-knot. That is, given a surface $\Sigma \subset X$ and a 2-knot $K \subset S^4$, we can form the surface $\Sigma_K$ as $(X,\Sigma_K) = (X,\Sigma) \# (S^4,K)$. Theorem \ref{t:main} tells us that $\Sigma$ and $\Sigma_K$ are concordant. Moreover, if $K$ is a ribbon 2-knot, or more generally a knot which is 0-concordant to the unknot, then $\Sigma_K$ will be 0-concordant to $\Sigma$, and hence the surgeries on $\Sigma$ are equivalent to the surgeries that can be done on $\Sigma_K$ by Theorem \ref{t:0concsurg}. In the case that $\Sigma$ is a torus, this means that if $X'$ is the result of a long transform on $\Sigma$, then there is a log transform on $\Sigma_K$ that also gives $X'$. However, if the complement of $\Sigma$ is simply connected, then $\Sigma$ and $\Sigma_K$ are smoothly equivalent to begin with by Theorem \ref{t:0conciso}.

While it is unknown whether all 2-knots in $S^4$ are 0-concordant to the unknot, we can ask a similar question about surfaces in a general 4-manifold. Are all concordant surfaces actually 0-concordant? The answer is no. In \cite{FSsurf}, Fintushel and Stern construct examples of surfaces which are topologically isotopic, have simply connected complement, but are not smoothly equivalent. By Theorem \ref{t:main}, they must be (smoothly) concordant, but if they were 0-concordant, they would be smoothly equivalent by Theorem \ref{t:0conciso}. On the other hand, this does not mean that non-trivial surgeries on smoothly distinct surfaces always result in non-diffeomorphic manifolds.


\section{Notation and Handlebodies}\label{notation}

Let $Y$ be an embedded submanifold in $X$. The tangent bundle of $X$ will be denoted $TX$ and the normal bundle of $Y$ in $X$ will be denoted $N_X Y$. A tubular neighborhood of $Y$ in $X$ will be denoted $\nu Y$ and often we will confuse this neighborhood with $N_X Y$. The restriction of $TX$ to $Y$ will be written as $TX|_Y$, while the restriction to the entire $n$-skeleton of $X$ will be written $TX|_n$.

\subsection {Handlebodies} Let $X$ be an $n$-manifold with boundary and divide the components of the boundary $\partial X$ into two sets, $\partial_-X$ and $\partial_+X$. By a relative handlebody decomposition of $(X,\partial_-X)$ we will mean a diffeomorphism 

\[  X = (\partial_-X \times I) + h_{a_1} + \ldots + h_{a_l}  \]

where $h^{a_i}$ is an $a_i$-handle and attaching is denoted additively. Handles are not necessarily attached in from lowest index to highest, although by standard transversaltiy results, we can rearrange their attaching to be thus. Such a handle decomposition is always induced from a Morse function which takes its minimum on $\partial_-X$ and its maximum on $\partial_+X$ which moreover has no critical points on a collar of $\partial_-X$. 

Very often in this paper we will need to consider handle decompositions on cobordisms between manifolds with boundary. These can be thought of as manifolds with corners. In general it will be simplest to think of such cobordisms as being defined by a handlebody decomposition   

\[  M = (X\times I) + h_{a_1} + \ldots + h_{a_l}  \] 

Where $X$ is a manifold with boundary, and the handles are attached to the the interior of $X\times \{1\}$. We will use the notation $\partial_-M$ and $\partial_+M$ for the manifolds representing the two ends of the cobordism (which may have boundary themselves). The key example from this paper is $X\times I -\nu Y$ where $Y$ is some concordance from $\Sigma_0$ to $\Sigma_1$ (which will always be a codimension-2 embedding). We can construct such a relative handlebody decomposition on the concordance complement by taking a Morse function that has its minimum on $X\times \{0\} - \nu \Sigma_0$, and has no critical points on a collar of $X\times \{0\} - \nu \Sigma_0$. 

In fact, if we know something about the embedding of $Y$, we can be even more explicit about how such a relative handlebody is constructed. In particular, the projection $X\times I \longrightarrow I$ induces (after a small perturbation of $Y$), a Morse function and therefore relative handle structure on $(Y,\Sigma_0)$. It is explained in \cite{GS} how such a handle decomposition on $Y$ induces a dual handle decomposition on $X\times I - \nu{Y}$. In particular, every $m$-handle in $Y$ contributes an $(m+1)$-handle to $X\times I - \nu{Y}$.

Since the fundamental group will be a constant consideration for us, let us explicitly describe what the dual decomposition of $X\times I-Y$ looks like when $Y$ has a 1-handle. For concreteness, let the induced relative handle structure on $(Y,\partial_-Y)$ in $X\times I$ consist of a single 1-handle, hence inducing a dual relative handle-structure on $X\times I - \nu Y$ as a single 2-handle attached to $(X - \nu \Sigma_0)\times I$. If the 1-handle of $Y$ is attached to $\partial_-Y \times I$ at point $a$ and $b$ in  $\partial_-Y \times \{1\}$, then the 2-handle in the dual decomposition is attached along meridians of $a$ and $b$ that have been band-summed together along the path of the 1-handle. This is described in detail in \cite{GS}. See Figure \ref{f:dual1handle}.

	\begin{figure}
\labellist
\small\hair 2pt
\pinlabel {a 1-handle in $Y$} at 162 13
\pinlabel {attaching sphere of 2-handle} at 316 133
\endlabellist	
\includegraphics{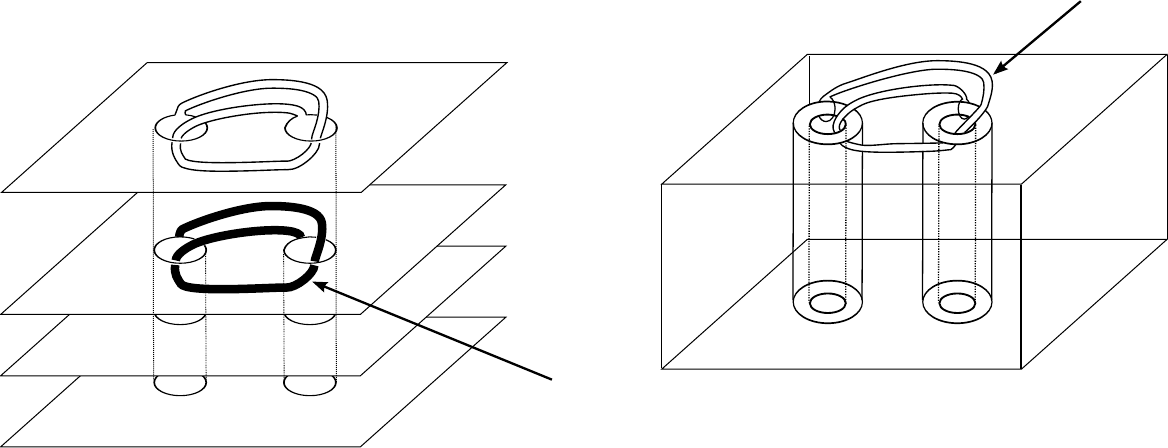}
	\caption{The dual handle decomposition on $X\times I - Y$ when $Y$ has a 1-handle}
	\label{f:dual1handle}
	\end{figure}

\subsection{Topological manifolds} Essentially all of the arguments in this paper also work in the topological category (except for those dealing with 0-concordance, which are strictly smooth results). In most cases in this paper, the quickest route to this fact is to note that a punctured 4-manifold always possesses a smooth structure (\cite{FQ}). Therefore we can make smooth arguments in the complement of a point. As an example, we can define a spin structure on a topological 4-manifold as being a spin structure in the complement of a point. Using this methodology, it will rarely be necessary to spell out in detail why a given proof works in the topological setting.

\section{Ambient surgery}\label{s:surgery1}

Most of the results in this paper rely on being able to modify codimension-2 embeddings through the process of ambient surgery, which we will now describe.

Given $Y^{n-2} \subset M^n$ a proper codimension-2 embedding, and let $\Delta^m$ be an embedded $m$-disk in $X$ which intersects $Y$ only along the boundary, that is $\Delta \cap Y = \partial \Delta$. Thicken $\Delta$ to an $n$-dimensional $m$-handle $h$ attached to $Y$. We can modify $Y$ by removing the attaching region of this $m$-handle, and replacing it by the complement of the attaching region in the boundary of $h$ (this is sometimes called the belt region, or the ``label''). See Figure \ref{f:ambientsurgery}. Call this surgered embedding $Y'$.

	\begin{figure}
\labellist
\small\hair 2pt
\pinlabel $\Delta$ at 11 97
\pinlabel $Y$ at 85 23
\pinlabel $Y'$ at 348 23
\endlabellist	
\includegraphics{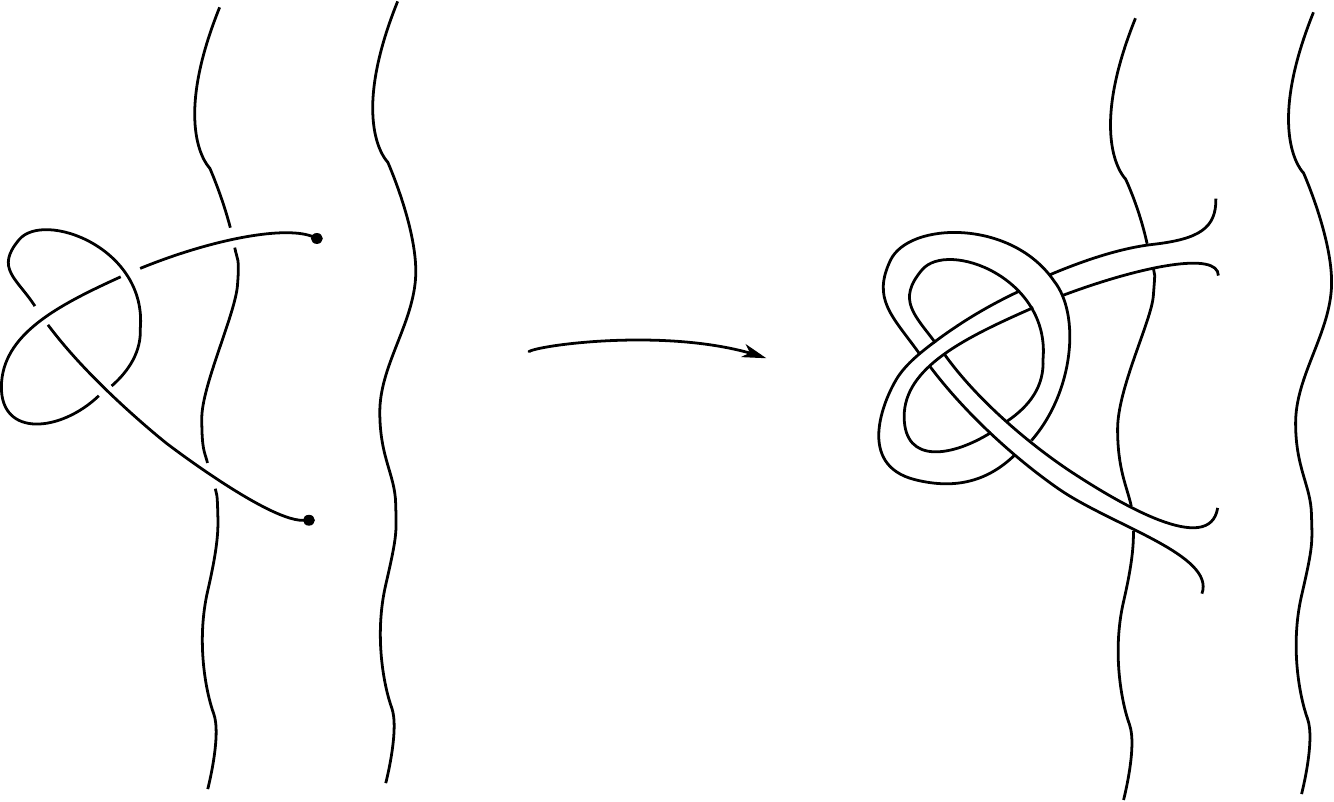}
	\caption{Ambient surgery along $\Delta$}
	\label{f:ambientsurgery}
	\end{figure}

Here are two examples of this. When $\Delta^1$ is one dimensional (an interval connecting two points of $Y$), the process of ambient surgery is just a self connect sum. In the case that $Y$ is 3-dimensional, $Y' = Y \# S^1\times S^2$.  When $\Delta^2$ is 2-dimensional, this kind of ambient surgery modifies a 3-dimensional $Y$ by Dehn surgery. This will be explored in more detail in the next section, where we determine what kinds of Dehn surgeries can be performed along a given $\Delta^2$.

\subsection{How ambient surgery affects the fundamental group}

 We frequently will want to know how $\pi_1(M - Y)$ is related to $\pi_1(M - Y')$, that is, how the fundamental group of the complement changes under ambient surgery. To do this, we'll view ambient surgery as a bordism from $Y$ to $Y'$ in $M\times I$ as follows. Begin with a codimension 2-embedding $Y\subset M$ and an $m$-disk $\Delta$ in $M$ along which we wish to perform ambient surgery. We will construct a $n-1$ manifold $\mathcal{Y}$ in $M\times I$ such that $\partial \mathcal{Y} = Y \times \{0 \}\sqcup Y' \times \{1\}$. The manifold $\mathcal{Y}$ is defined as $Y\times I \cup h \times \{1\}$ in $M\times I$ where $h$ is the $m$-handle that determines the surgery. See Figure \ref{f:ambientsurgerybordism}.

	\begin{figure}
\labellist
\small\hair 2pt
\pinlabel {$Y$} at 193 65
\pinlabel {$Y'$} at 193 253
\pinlabel {$h_m$} at 257 211
\pinlabel {$\mathcal{Y}$} at 433 161
\pinlabel {$M \times \{ 0\}$} at 12 25
\pinlabel {$M \times \{ 1\}$} at 12 219
\endlabellist	
\includegraphics[width=\textwidth]{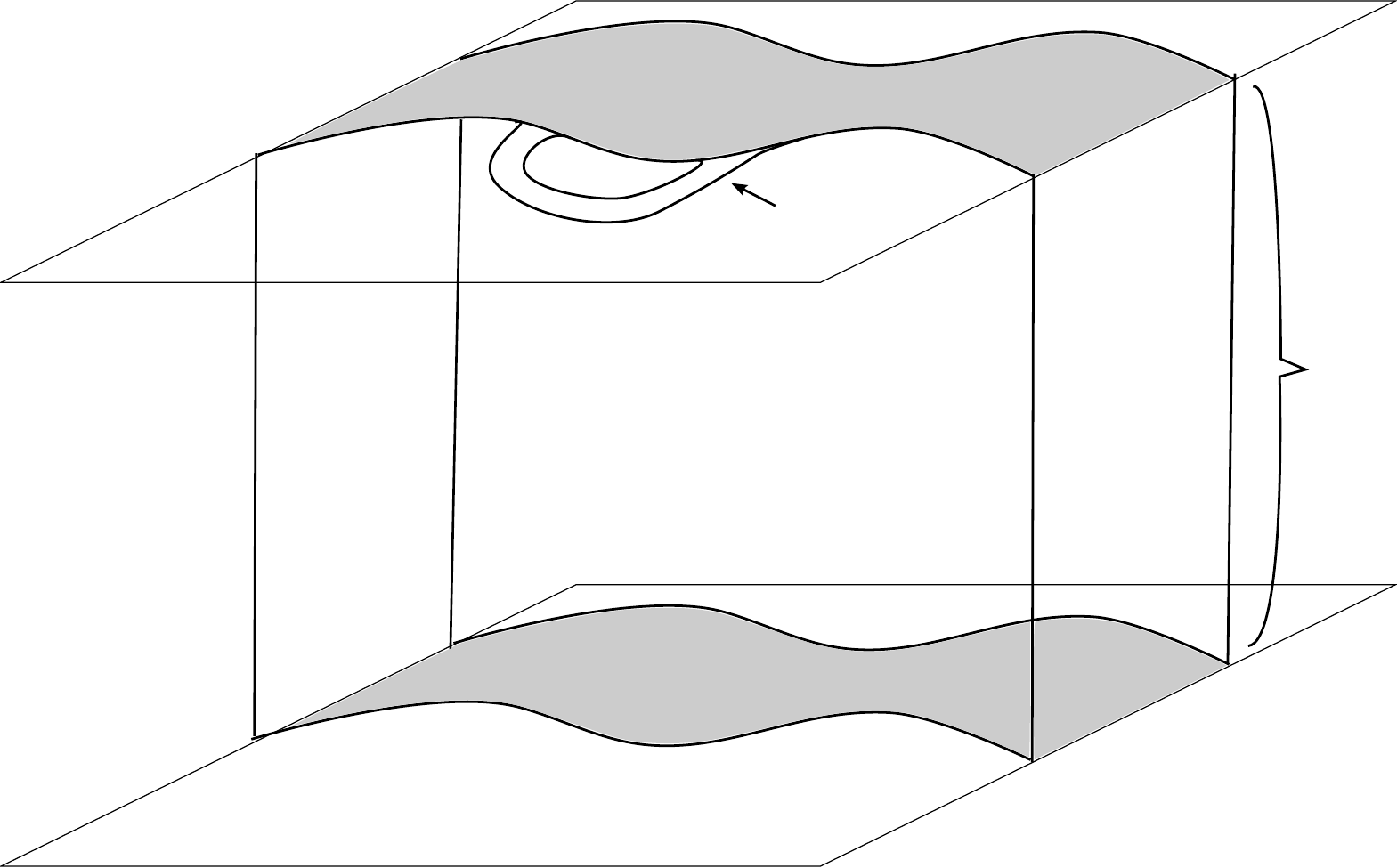}
	\caption{Viewing ambient surgery as a bordism}
	\label{f:ambientsurgerybordism}
	\end{figure}
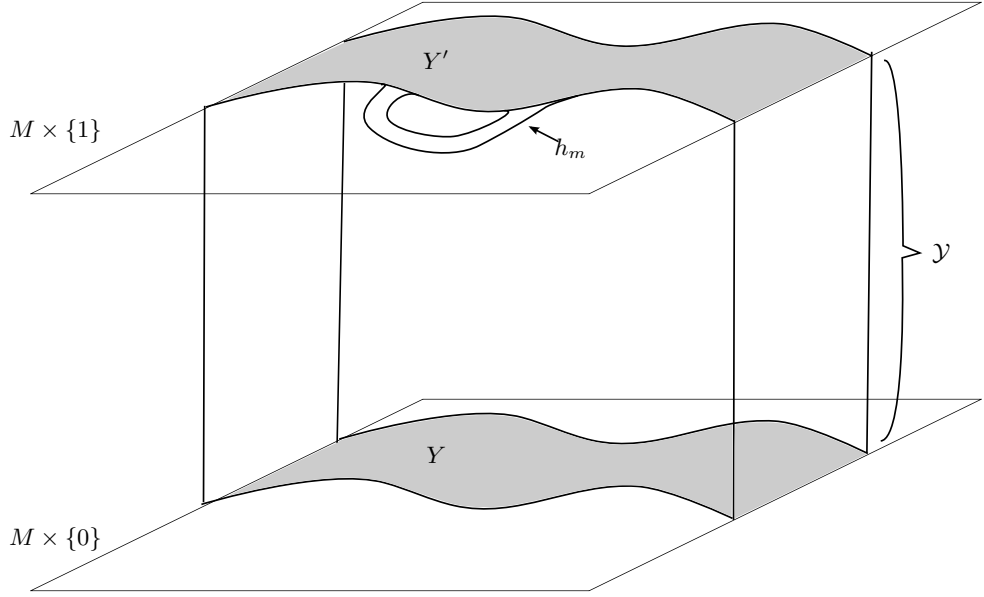

To determine the fundamental group of $M - Y'$, we'll construct a handle decomposition of $M\times I - \nu \mathcal{Y}$. Following the construction in Section \ref{notation}, a handle structure on $\mathcal{Y}$ induces a dual handle structure on $M\times I - \nu \mathcal{Y}$. This relative handle structure induced on $\mathcal{Y}$ is the evidentially just $Y\times I \cup h_m$. See Figure \ref{f:ambientsurgerybordism}. That is, ambient $m$-surgery corresponds to a handle decomposition on $M\times I - \nu\mathcal{Y}$ as

\[ M\times I - \nu\mathcal{Y} = (M - \nu Y)\times I + h_{m+1} \]
where $h_{m+1}$ is an $(n+1)$-dimensional $(m+1)$-handle.

\begin{proposition}\label{p:pi1surgery2}

Let $Y^3 \subset M^5$ be a proper embedding. Then ambient 2-surgery (ambient Dehn surgery), does not change the fundamental group. That is, $\pi_1(M - Y)$ is isomorphic to $\pi_1(M - Y')$.\footnote{More generally, the same proof applies to the case that $Y^{n-2} \subset M^n$ is a proper codimension 2 embedding, and we perform ambient i-surgery for $2\leq i \leq n-3$.}
\end{proposition}
\begin{proof}

We have seen that ambient 2-surgery corresponds to a handlebody decomposition  
\[ M\times I - \nu \mathcal{Y} = (M-\nu Y)\times I + h_{3} \] where $\partial_-(M\times I - \nu \mathcal{Y}) = M - \nu Y$. Therefore, we get an isomorphism between $\pi_1(M\times I - \nu \mathcal{Y})$ and $\pi_1(M - \nu Y)$.

Turning this construction upside down, we apply the same proof to \[ M\times I - \nu \mathcal{Y} = (M - \nu Y')\times I + h'_{3} \] to see that $\pi_1(M\times I - \nu \mathcal{Y})$ is isomorphic to  $\pi_1(M - \nu Y')$
\end{proof}

The purpose of performing ambient 1-surgeries will be to set ourselves up to perform ambient 2-surgeries. The following tells us when ambient 2-surgeries are possible.

\begin{lemma}\label{l:2surgispos} Let $Y^{n-2}$ be a proper codimension-2 embedding in $M^n$ $(n \geq 5)$ whose complement has cyclic fundamental group generated by the meridian to $Y$. Then any loop in $Y$ bounds an embedded disk in $M$ that only intersects $Y$ along the boundary.\end{lemma}

\begin{proof} 

It is sufficient to find a push-off of the loop that is null-homotopic in $M - Y$. Since the fundamental group of $M-Y$ is generated by a meridian to $Y$, however, any push-off can be modified by adding copies of the meridian to make the push-off null-homotopic.
\end{proof}

We can achieve this situation as follows. 
\begin{proposition}\label{p:1surgery}
Let $Y^{n-2} \subset M^n$ be a proper codimension-2 embedding where $M$ is simply connected. Then by performing ambient 1-surgeries to $Y$, we can make the fundamental group of the complement a cyclic group. 
\end{proposition}

\begin{proof}
Suppose we intend to perform ambient 1-surgery along an arc $\Delta^1$ in X with endpoints $a,b\in Y$. An ambient 1-surgery corresponds to a handle decomposition

\[ M\times I - \nu \mathcal{Y} = (M-\nu Y)\times I + h_{2}. \]

We saw in Section \ref{notation} that $h_2$ is attached along the band sum of meridians of $a$ and $b$ band-summed together along $\Delta^1$.

Notice that $\pi_1(M-Y)$ is generated by conjugates of a meridian to $Y$. Therefore, we can use ambient 1-surgeries to identify all of the generators. Since a 1-generator group is cyclic, we are done.

\end{proof}

This will show up in several spots in our proof of Theorem \ref{t:main}. It allows us to set up to perform ambient 2-surgeries, and it also implies that every homology class in a simply connected 4-manifold can be represented by an embedded surface whose complement has cyclic fundamental group.

\section{Spin structures and surgery}\label{s:surgery2}

In this section we will focus on ambient 2-surgeries. Ambient 1-surgeries can be performed in an essentially unique way across a given arc. In contrast, there are many ambient 2-surgeries that can be performed over any given disk. The purpose of this section is to gain control over which ambient 2-surgeries can be performed using spin structures.

For the purposes of this paper, we will define a \emph{spin structure} on a manifold to be a trivialization of the tangent bundle over the 1-skeleton that can be extended across the 2-skeleton.\footnote{For one and two dimensional manifolds, it is conventional to consider a spin structure as a trivialization of the \emph{stable} tangent bundle over the 1-skeleton that extends over the 2-skeleton.} We can extend this definition to topological 4-manifolds using the fact that every 4-manifold possesses a smooth structure in the complement of a point. Possessing a spin structure is equivalent to $w_2 = 0$, and for simply connected 4-manifolds, to the intersection form being even. Additionally, 3-manifolds always admit spin structures because they are parallelizable. Spin structures are parameterized by $H^1(X,\mathbb{Z}_2)$. For a much more thorough review of spin structures from a similar perspective, see \cite{Kirbybook} or \cite{Scorpbook}.

\subsection{Spin Surgery}
 In this paper, when we speak of 2-surgery, or Dehn-surgery on a 3-manifold, we will always mean integral Dehn surgery, or equivalently, the result of attaching a 4-dimensional 2-handle to $Y^3 \times I$. A surgery is called a \emph{spin surgery} when the spin structure on $Y\times I$ extends across this 2-handle. It is well known that the spin surgeries on $S^3$ are equivalent to attaching 2 handles to $S^3$ with even framing. One can think of a spin surgery as being the generalization of even surgeries to 3-manifolds other than $S^3$. It is a standard fact about 3-manifold that $\Omega^{spin}_3 = 0$, i.e. all 3-manifolds are related by spin surgeries, (see e.g. \cite{Kirbybook}).

\subsection{Ambient spin-surgery}

Ambient 2-surgery is the most subtle sort of surgery we will encounter in this paper. To perform 2-surgery on a 3-dimensional manifold, one must specify a framing for the surgery. The most troublesome question is to understand which framings admit ambient surgeries and which do not. In general, not all framings do, but we will see that with the extra framework which spin structures afford, we can always perform spin surgeries ambiently. In particular, we will see that a spin structure on $X^5$ induces one on $Y^3$, and that ambient 2-surgeries can be performed on $Y$ as long as they agree with this induced spin structure. Specifically, given a spin structure on $Y$, (i.e. a trivialization of the $TY|_1$ that extends over $TY|_2$), and a disk $\Delta \subset X$ along which we wish to perform ambient surgery, we will call an ambient 2-surgery a \emph{spin surgery} if the induced trivialization on $TY|_1 \oplus N_\Delta \partial \Delta$ extends over the 2-handle specifying the surgery. This differs from spin surgery defined above because we the surgery takes place along a 2-handle directly abutting $Y$, rather than as a 2-handle attached to $Y\times I$. The two constructions correspond by thinking of the spin structure on $Y \times I$ as being a product of a spin structure on $Y$ with a line bundle. The role of the line bundle is played in the ambient surgery case by $N_\Delta \partial \Delta$. That is, we do not need to specify the line bundle ambiently on all of $Y$ (there may not be a trivial line bundle in $N_XY$ over all of Y), but only over the curve $\partial\Delta$ where the surgery is performed.

Inducing an appropriately compatible spin structure on $Y$ can be rather delicate, and in general we will need to use ambient spin surgery even in the case that $X$ is not spin.

\begin{proposition}\label{p:3hasspin}
Let $Y^3$ be an oriented 3-manifold embedded in $X^5$ such that $\langle w_2(X-Y), h \rangle = 0$ for all $h\in H_2(X-Y,\mathbb{Z})$. Then there exists a spin structure on $Y$ such that all ambient surgeries are spin surgeries.
\end{proposition} 

\begin{proof}

We'll first trivialize $TY$ over embedded 1-cycles. Let $\gamma = S^1$ be an embedded 1-cycle in $Y$. There are two cases. First, suppose $\gamma$ has a push-off into $X-Y$ that is null homologous in $X-Y$. Then $\gamma$ bounds an oriented surface $\Delta$ whose interior lies in $X-Y$. This push-off specifies a line bundle $\ell$ as a subbundle of $N_XY|_\gamma$. Choose further a 2-plane subbundle $\delta$ of $N_X\Delta$ that extends $N_Y\gamma$. See Figure \ref{f:spinsurgery}.

	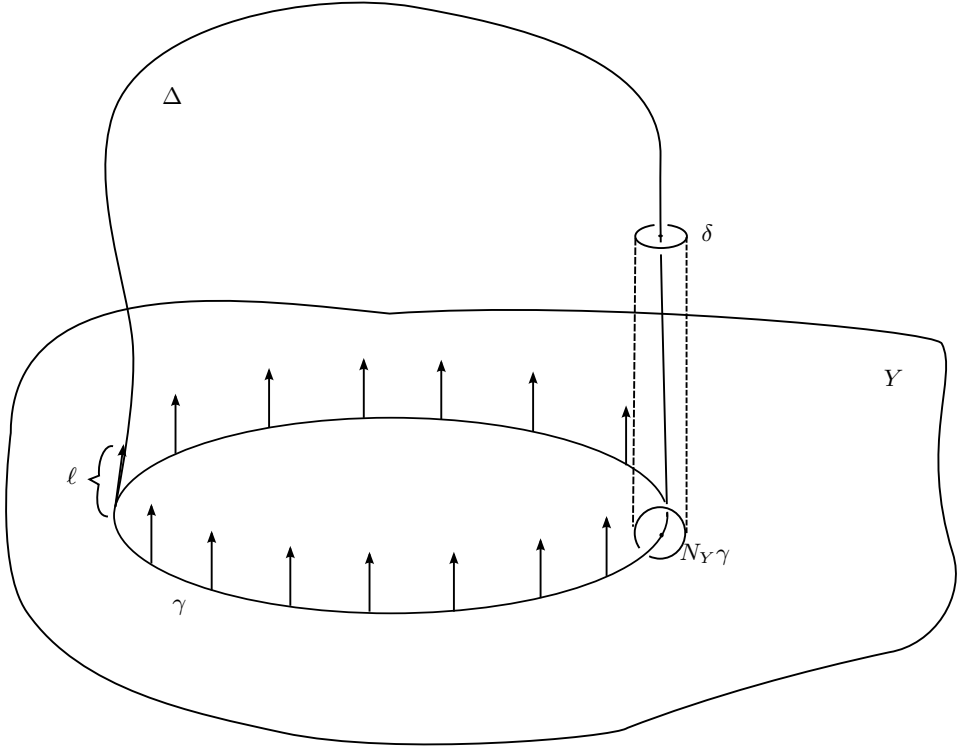
\begin{figure}
\labellist
\small\hair 2pt
\pinlabel $\ell$ at 28 113
\pinlabel $Y$ at 372 154
\pinlabel $\Delta$ at 70 272
\pinlabel {$N_Y \gamma$} at 293 80 
\pinlabel $\delta$ at 294 215
\pinlabel $\gamma$ at 73 58
\endlabellist	
\includegraphics[width=\textwidth]{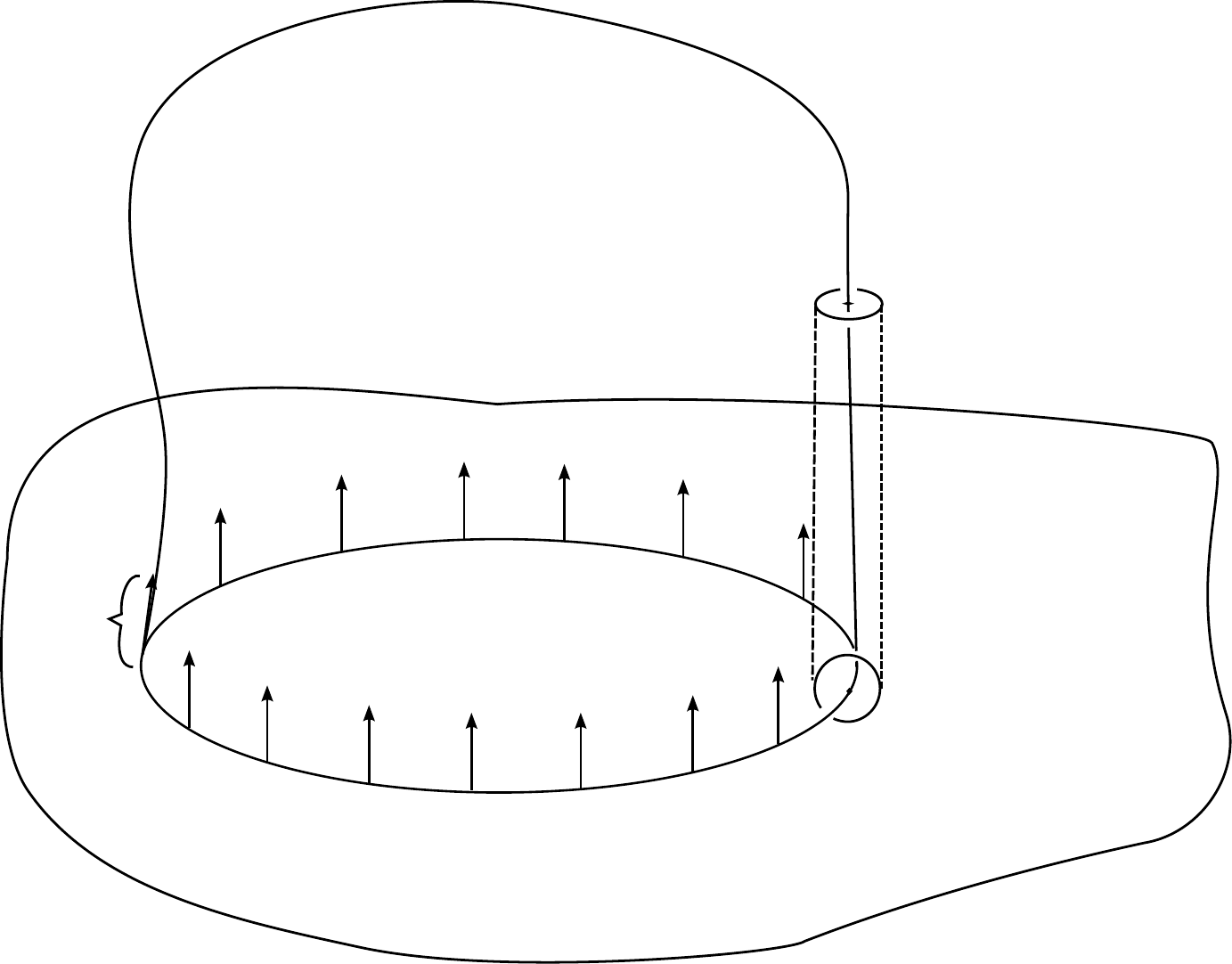}
	\caption{Trivializing $TY|_\gamma$}
	\label{f:spinsurgery}
	\end{figure}

Finally, we can define a trivialization of $TY|_\gamma$ as follows. Trivialize it in such a way that the induced trivialization on $TY|_\gamma \oplus \ell$ extends across $\delta \oplus T\Delta$. (Once we show that these trivializations actually define a spin structure on $Y$, it is evident that the disk bundle of $\delta$  defines a 2-handle giving an ambient spin surgery along $\Delta$).

The second case to consider is when $\gamma$ does not have a null-homologous push-off into $X-Y$. In this case we can trivialize $TY|_\gamma$ arbitrarily.\footnote{In the case that X is spin, we could requiring that the spin structures on Y and X somehow agree over homologically essential loops, however we will not be able to perform ambient surgery along such loop anyways, so our choice of trivialization does not matter.}

It is necessary to verify the following three facts about the trivializations we have defined. 

\begin{enumerate}
\item The trivialization of $TY|\gamma$ does not depend on our choice of $\Delta$ or our choice of $\delta$. This shows our trivialization is well defined on the 1-cycles.

\item If $\gamma_1$  and $\gamma_2$ are equal in $H_1(Y)$, then the induced trivializations of $TY|\gamma_i$ extend across all oriented surfaces $\Sigma \subset Y$ such that $\partial \Sigma = \gamma_1 \cup \gamma_2$. This shows that our construction defines a trivialization of $TY|_1$, and not just trivializations over the 1-cycles.

\item If $\gamma$ is null homologous in $Y$, then the trivialization of $TY|_\gamma$ extends over all surfaces bounded by $Y$. This shows that we have in fact defined a spin structure on $Y$.

\end{enumerate}

(All three of these facts essentially are due to the fact that a trivial 5-plane bundle cannot have any non-trivial oriented 4-plane sub-bundles.)

To show (1), let $\Delta_1$ and $\Delta_2$ be disks with boundary $\gamma$, and let $\delta_1$ and $\delta_2$ be 2-plane bundles in $N_X \Delta_1$ that extend $N_Y\gamma$.

If we suppose that these choices induce different trivializations of $TY|_\gamma$, then $\delta_1 \oplus T\Delta_1 \cup \delta_2 \oplus T\Delta_2$ is an oriented 4-plane bundle over $\Delta_1 \cup \Delta_2$ with non-trivial $w_2$. Since, after a small perturbation, this surface lies in $X-Y$, such a bundle contradicts our condition on $w_2(X-Y)$.

To show (2), suppose $TY|_{\gamma_1\cup \gamma_2}$ has been given a trivialization that does not extend over $TY|_\Sigma$. Let $\overline{\ell}$ be a line bundle in $N_X Y|_\Sigma$, specifying a push-off of $\Sigma$ into $X-Y$ such that the push-offs of the $\gamma_i$ are null-homologous in $X-Y$. This exists because $\Sigma$ has the homotopy type of a 1-complex. Finally, let $\Delta_1$ and $\Delta_2$ be oriented surfaces bounding $\gamma_1$ and $\gamma_2$ whose interiors lie in $X-Y$, and let $\delta_1$ and $\delta_2$ be 2-plane bundles in $N_X\Delta_1$ and $N_X\Delta_2$ used to trivialize $TY|_{\gamma_1\cup \gamma_2}$. Then we can form a 4-plane bundle over the oriented surface $\Delta_1 \cup \Sigma \cup \Delta_2$ as $(T\Delta_1 \oplus \delta_1) \cup TY|_\Sigma \oplus \overline{\ell} \cup (T\Delta_1 \oplus \delta_2)$. However, this surface can be perturbed into $X-Y$, and since this 4-plane bundle is non-trivial, we again contradict the assumption on $w_2(X-Y)$.

Finally, (3) follows immediately from (2) when we let $\gamma_2$ be the empty set.

\end{proof}

\begin{lemma}
Let $Y^3$ be an embedded submanifold in $X^5$, and $\gamma$ be a loop in $Y$ which bounds a disk $\Delta$ in $X-Y$. If $Y$ possesses a spin structure such that there exists an ambient spin surgery that can be performed across $\Delta$, then \emph{all} spin surgeries can be performed across $\Delta$.
\end{lemma}

\begin{proof}
An ambient surgery across $\Delta$ is specified by a framing of $N_Y\gamma$ that extends across a 2-plane subbundle of $N_X\Delta$. Since we are looking for the 2-plane bundles in $N_X\Delta$, (which is a trivial 3-plane bundle), the obstruction to extending a framing on $N_Y\gamma$ across $N_X\Delta$ lies in $\pi_1(V_2(\mathbb{R}^3)) = \mathbb{Z}_2$. This obstruction coincides with the whether or not a given framing of $N_Y\gamma$ (specifying an ambient 2-surgery) induces a framing on $N_Y\gamma \oplus T\gamma$ that agrees or disagrees with the spin structure on $Y$.

\end{proof}

\begin{cor}[\cite{Hir}] Every orientable 3-manifold embeds in $\mathbb{R}^5$.
\end{cor}
\begin{proof} $S^3$ can be embedded in the boundary of $D^5$. Then  ambient spin surgeries on this $S^3$ can be performed arbitrarily, since any collection of loops in $S^3$ bounds a disjoint collection of embedded disks in $D^5$. Since $\Omega^{spin}_3 = 0$, all 3-manifolds are related by spin surgeries, therefore $S^3$ can be modified by ambient spin surgeries to be any manifold we choose.
\end{proof}
 This is, essentially, Hirsch's original proof of this fact, \cite{Hir}. Proofs showing that non-orientable 3-manifolds also embed in $\mathbb{R}^5$ using the same set of ideas are given in \cite{Wallemb}, \cite{R}, and \cite{BS}.

\subsection{Spin-surgery on 3-manifolds with boundary}
Given a 3-manifold $Y$ with boundary, what 3-manifolds can be obtained by spin surgeries on the interior of $Y$? The situation is much the same as the closed case. We have the following generalization of the fact that $\Omega^{spin}_3 = 0$.

\begin{lemma} 
\label{l:partialspinsurgery} Any two spin 3-manifolds with boundary are related via spin-surgeries if the spin structures agree on the boundary.
\end{lemma}
We  will use ambient spin-surgeries on 3-manifolds with boundary to construct concordances between surfaces.

\begin{proof}
Let $Y_0$ and $Y_1$ be 3-manifolds with boundary $\partial Y_i = \Sigma_i$ and let $\phi:\Sigma_0 \longrightarrow \Sigma_1$ be a spin homeomorphism of the boundary which takes the spin structure of $\Sigma_0$ to that of $\Sigma_1$.
It will be sufficient to construct a spin-cobordism rel boundary from $Y_1$ to $Y_2$ that has a handle decomposition without 1-handles. 
Begin by constructing a closed spin 3-manifold $Y$ by gluing the boundaries together of $Y_1$ and $Y_2$ together via $\phi$, e.g. $Y = Y_0 \cup_{\phi \times \{0\}} (\Sigma_1 \times I) \cup_{id \times \{1\}} Y_1$. Since $\Omega
^{spin}_3 = 0$, we know that $Y$ bounds a spin 4-manifold $X$. Induce a relative handle decomposition on $(X,Y_1,Y_2)$ from a Morse function $f: Y \rightarrow [0,1]$ such that $f(Y|_{Y_0}) = 0$, $f(Y|_{Y_1}) = 1$, and $f(Y|_{\Sigma\times \{t\}}) = t$. Surger out the 1-handles. This gives the required cobordism.

\end{proof} 

Note that the boundary did not need to be connected.

\begin{cor}[\cite{Ker}]\label{ss:KerSlice} Every 2-knot $K$ in $S^4$ is slice, that is, bounds a $D^3$ in $D^5$. \end{cor}
\begin{proof}
Perform ambient spin surgeries in $D^5$ to a Seifert manifold for $K$.
\end{proof}
 In fact, we can similarly show that \emph{every} punctured 3-manifold is a Seifert surface for $K$ in $D^5$. This is essentially Kervaire's proof. 

\section{Concordance of Surfaces in a 4-manifold}\label{s:concordance}

Having developed in the previous two sections the necessary tools for modifying codimension 2 submanifolds, it will now be a relatively simple thing to determine the concordance classes of surfaces in 4-manifolds. It is clearly necessary for concordant surfaces to be in the same homology class. The following says that this is also sufficient.

\begin{theorem}\label{t:main}
Let X be a simply connected manifold. Then two smoothly (resp. locally flat) embedded surfaces  $\Sigma_0$ and $\Sigma_1$ are smoothly (resp. locally flat) concordant if and only if they have the same genus and are in the same homology class.
\end{theorem}

The basic structure of the proof is as follows:
\begin{itemize}
\item Find \emph{some} 3-manifold in $X\times I$ connecting $\Sigma_0 \times \{0\}$ and $\Sigma_1 \times \{1\}$.
\item Use Proposition \ref{p:1surgery} to make the fundamental group of the complement cyclic.
\item Use ambient 2-surgery as in Section \ref{s:surgery2} to surger the 3-manifold to a concordance.
\end{itemize}

\begin{proof}
Since $\Sigma_0$ and $\Sigma_1$ represent the same homology class (and since both possess normal bundles), they are the preimage of $CP^1$ of homotopic maps in $[X,CP^2] = H_2(X)$. Therefore, by the Thom construction there is a 3-manifold $Y$ embedded in $X \times I$ such that $\partial Y = \Sigma_0 \times \{0\}$ and $\Sigma_1 \times \{1\}$. 

By Proposition \ref{p:1surgery}, we can perform ambient 1-surgery on $Y$ so that $\pi_1(X\times I-Y)$ is cyclic. By Lemma \ref{l:2surgispos} this implies we can perform ambient 2-surgery on any loop in $Y$. There are two cases. First suppose that $\langle w_2(X\times I-Y), h \rangle = 0$ for all $h\in H_2(X\times I-Y,\mathbb{Z})$. Then by Proposition \ref{p:3hasspin} we can give $Y$ a spin structure such that we can perform any spin-surgery along any loop in $Y$. Lemma \ref{l:partialspinsurgery} tells us that $Y$ is related by spin surgeries to a product $\Sigma\times I$, so we can perform ambient spin surgeries, to make $Y$ into a concordance.

In the second case, suppose there is no induced spin structure on $Y$, that is suppose $\langle w_2(X\times I-Y), h \rangle \neq 0$ for some $h\in H_2(X\times I-Y,\mathbb{Z})$. The class $h$ can be represented by an embedded sphere $S$ because of the following exact sequence
\[ \pi_2(X\times I -Y) \longrightarrow H_2(X\times I - Y) \longrightarrow H_2(\pi_1) \longrightarrow 0 \]
and the fact that $H_2(G) = 0$ when $G$ is a cyclic group. Now we can surger $Y$ to a concordance for the following reason. We can perform \emph{any} 2-surgery ambiently on $Y$, not just the ambient spin surgeries, because if $\Delta$ is a disk with $\partial \Delta = \gamma \subset Y$ along which we would like to perform ambient 2-surgery, and we cannot perform the surgery we would like to over $\Delta$ (that is, we have a trivialization of $N_Y\gamma$ that does not extend over $N_X\Delta$), then after tubing $\Delta$ to a copy of $S$, we will be able to perform that surgery.

\end{proof}

\begin{remark} We in fact show much more than that $\Sigma_0$ and $\Sigma_1$ are concordant. The same proof shows that for any two homologous surfaces in X, one can show that not only does a concordance exist (in the case when the surfaces have the same genus), but also that \emph{any} 3-manifold with the correct boundary can be substituted in for the concordance. 
\end{remark}

\begin{remark}
With only minor modifications it is also possible to show that concordances can be found in any h-cobordism, not just in $X\times I$. That is if $M^5$ is an h-cobordism with $\partial M = X \cup \tilde{X}$, then there exists an embedded $Y = \Sigma \times I$ connecting surfaces in $X$ and $\tilde{X}$ that have the same genus and homology class in $M$.
\end{remark}

We can also prove a non-simply connected version. 

\begin{theorem}\label{t:pi1conc}
Let $\Sigma_0$ and $\Sigma_1$ be two surfaces which are smoothly (resp. locally flat) embedded in a not-necessarily simply connected 4-manifold X. Assume further that the following are satisfied:
\begin{enumerate}
\item $\pi_1(\Sigma_i) \longrightarrow \pi_1(X)$ is trivial for $i=0,1$.
\item $[\Sigma_0]= [\Sigma_1]$ in $H_2(X,\mathbb{Z}[\pi_1])$.
\item There exists a third surface $\overline{\Sigma}$ in $X$ such that $\pi_1(\overline{\Sigma}) \longrightarrow \pi_1(X)$ is trivial, $[\overline{\Sigma}]= [\Sigma_i] \in H_2(X,\mathbb{Z}[\pi_1])$, and the meridian of $\overline{\Sigma}$ is null-homotopic in $X - \overline{\Sigma}$
\end{enumerate}
Then $\Sigma_0$ and $\Sigma_1$ are smoothly (resp. locally flat) concordant.
\end{theorem}


\begin{proof}
Condition (1) implies that there is a lift of $\Sigma_i \hookrightarrow X$ to the universal cover $\Sigma_i \hookrightarrow \tilde{X}$. Call these lifts $\tilde{\Sigma}_0$ and $\tilde{\Sigma}_1$. Condition (2) implies that these lifts can be taken to be homologous in $\tilde{X}$. Since $\tilde{\Sigma}_0$ and $\tilde{\Sigma}_1$ are homologous, there is a 3-manifold $Y$ in $\tilde{X}\times I$ connecting them, and by condition 3, this $Y$ can be taken to have a simply connected complement. (Find manifolds $Y_1$ and $Y_2$ in $\tilde{X}\times I$ going from $\Sigma_0$ to $\overline{\Sigma}$ and $\overline{\Sigma}$ to $\Sigma_1$ respectively. The composition of these manifolds is $Y$, and since $\overline{\Sigma}$ has null-homotopic meridian, the complement of $Y$ will be simply connected).

Following exactly the proof of Theorem \ref{t:main}, we can modify $Y$ to be a concordance from $\tilde{\Sigma}_0$ to $\tilde{\Sigma}_1$ in $X \times I$. This descends to an immersed concordance in $X\times I$. The rest of the proof will be dedicated to showing that we can modify Y in such a way that it is disjoint from all of its deck transformations, and hence descends to an \emph{embedded} concordance in $X \times I$.

Let $\{\Sigma_{i,\alpha}\}_{\alpha \in \pi_1(X)}$ denote the lifts of the $\Sigma_i$ into $\tilde{X}$, and let $\{Y_{\alpha} \}_{\alpha \in \pi_1(X)}$ denote the collection of manifolds connecting $\Sigma_{0,\alpha}$ to $\Sigma_{1,\alpha}$. Each $Y_{\alpha}$ is embedded, (although the collection $\cup Y_{\alpha}$ may be immersed), and the deck transformations act transitively on the $Y_{\alpha}$

\begin{lemma}
The $Y_{\alpha}$ can be surgered (equivariantly) so that $\cup Y_{\alpha}$ is embedded in $\tilde{X}\times I$.

\end{lemma}

\begin{proof}

Suppose the collection $\cup Y_{\alpha}$ is immersed in $\tilde{X}\times I$. The double points of the immersed points will be a collection of $S^1$'s. Note that the complement of $\cup Y_{\alpha}$ is simply connected, so we can perform ambient surgery along $\cup Y_{\alpha}$. Choose a particular $Y_{\alpha}$, and perform ambient Dehn surgery to it (avoiding the double points, and avoiding the other components of $\cup Y_{\alpha}$) so that the $S^1$ double points in $Y_{\alpha}$ bound disjoint disks in $Y_{\alpha}$ (See Figure \ref{f:immersedsurgery}). Do this equivariantly to all of the $Y_{\alpha}$, so that all double points bound disjoint disks in some $Y_{\alpha}$. Let $D$ be one such disk in a given $Y_{\alpha}$ that has as its boundary the intersection between $Y_{\alpha}$ and $Y_{\alpha'}$. Choose a 2-plane subbundle of $N_XY_{\alpha}|_D$ that extends $N_{Y_{\alpha'}}\partial D$. The disk bundle of this 2-plane bundle specifies a surgery to $Y_{\alpha'}$ (similar to ambient 2-surgery). Such a surgery removes the double point set. This can be done with all of the double points (equivariantly on all of the $Y_{\alpha}$) because all of the double-points bound disjoint disks. See Figure \ref{f:resolveimmersion}.

\end{proof}

	\begin{figure}
\labellist
\small\hair 2pt
\pinlabel {$Y_{\alpha} $} at 23 80
\pinlabel {$Y_{\alpha'}$} at 155 79

\endlabellist	
\includegraphics{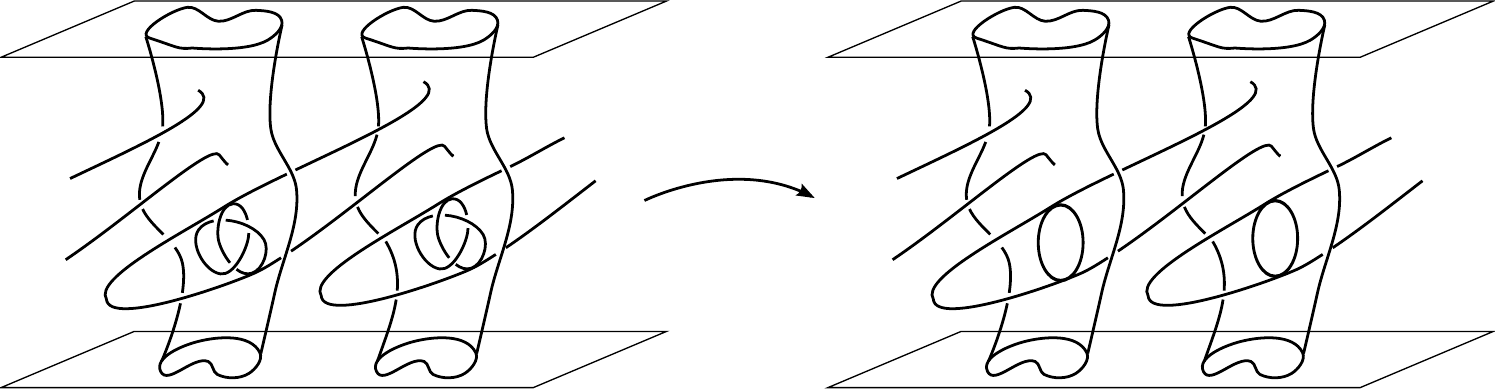}
	\caption{Unknotting the double-points}
	\label{f:immersedsurgery}
	\end{figure} 
 
 	\begin{figure}
\labellist
\small\hair 2pt
\pinlabel $Y_{\alpha'}$ at 49 80
\pinlabel $Y_{\alpha}$ at 28 36
\pinlabel {$N_{Y_{\alpha'}}\partial D$} at 162 14
\pinlabel $D$ at 99 49
\pinlabel {$N_X Y_{\alpha}|_D$} at 359 27
 \endlabellist
 	
\includegraphics{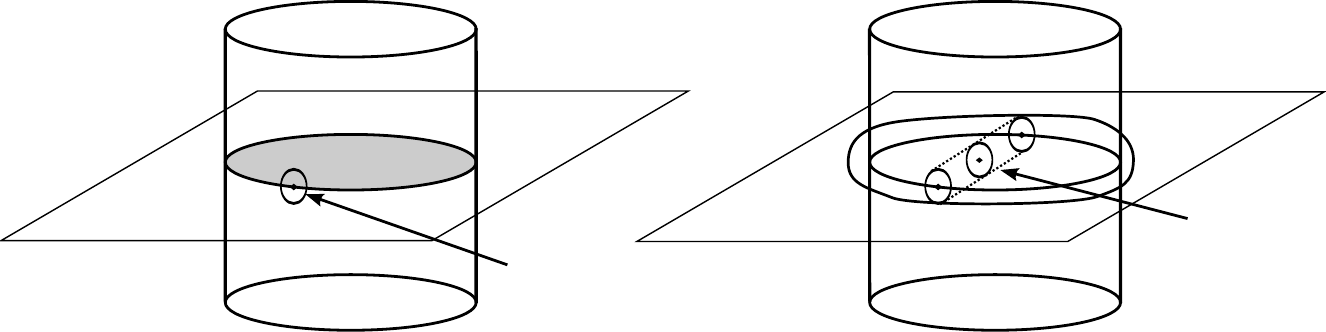}	
	\caption{Eliminating the double points.}
	\label{f:resolveimmersion}
	\end{figure}

Now that we have embedded $Y_{\alpha}$ on which the deck transformations act transitively, we can perform surgery (in a way that is equivariant under the deck transformations) to make the $Y_{\alpha}$ into concordances. Note that ambient 2-surgeries are supported in the neighborhood of 2-disks in a 5-manifold, therefore they can easily be made disjoint from their deck transformations. The $Y_{\alpha}$ then descend to a concordance in $X\times I$.

\end{proof}

\section{Surfaces up to isotopy}\label{s:isotopy}

In this section we would like to investigate when two homologous surfaces are more than concordant. We want to decide when they will actually be topologically isotopic. We'll see that this is often true when the complement of each surface has cyclic fundamental group. In everything that follows, $\Sigma_0$ and $\Sigma_1$ are locally-flat embedded surfaces of the same genus. There are three cases which will have somewhat different features: the simply connected case ($[\Sigma_i]$ is primitive), the finite cyclic case ($[\Sigma_i]$ is non-trivial), and the infinite cyclic case ($[\Sigma_i]=0$).

\begin{theorem}\label{t:iso1}
If X is a simply connected manifold, and $\Sigma_0$ and $\Sigma_1$ are in the same homology class, and moreover $\pi_1(X - \Sigma_i) =0$, then the surfaces are topologically isotopic.
\end{theorem}

\begin{proof}
By Theorem \ref{t:main}, $\Sigma_0$ and $\Sigma_1$ must be concordant, and by Meyer-Vietoris the complement of this concordance must be an h-cobordism. Since this h-cobordism is simply connected, by Freedman's theorem \cite{FQ}, it is a product, and we can conclude that $X-\nu\Sigma_0$ is homeomorphic to $X-\nu\Sigma_1$. This can be extended to give a homeomorphism of pairs from $(X,\Sigma_0)$ to $(X,\Sigma_1)$, which moreover induces the identity map on homology. By a theorem of Quinn, such a homeomorphism must be isotopic to the identity (\cite{Quinn}).

\end{proof}

The next simplest case is when $\pi_1(X - \Sigma_i) = \mathbb{Z}$. 

\begin{theorem}\label{t:infiso}
Let $X$ be a simply connected manifold with $b_2 \geq |\sigma| + 6$, and let $\Sigma_0$ and $\Sigma_1$ be embedded surfaces in the same homology class with $\pi_1(X- \Sigma_i) = \mathbb{Z}$. Then the surfaces are topologically isotopic. 
\end{theorem}

\begin{proof}

Without loss of generality, we can assume that $\Sigma_0$ bounds a solid handlebody in $X$, that is $\Sigma_0$ is a standardly embedded unknotted surface. Our fundamental group assumption implies that both $\Sigma_0$ and $\Sigma_1$ are null-homologous. Therefore by Theorem \ref{t:main} there is a concordance $Y \subset X\times I$ connecting the surfaces. The concordance $Y$ constructed has $\pi_1(X\times I - Y) = \mathbb{Z}$. This follows from Propositions \ref{p:1surgery} and \ref{p:pi1surgery2}. 

Consequently $\pi_1(X\times \{i\} - \Sigma_i)\longrightarrow \pi_1(X\times I - Y)$ is an isomorphism, and we can conclude that $X\times I - \nu Y$ has a handle decomposition with only 2 and 3-handles. This follows from the standard handle cancellation proof of the s-cobordism theorem. Moreover, the 2-handles must be attached along null homotopic loops in $X- \nu \Sigma_0$ (otherwise the map on the fundamental group would be only a surjection, not an isomorphism).

Turning the handlebody picture upside down, we have a handlebody whose 2-handles are attached along null-homotopic loops in  $X- \nu \Sigma_0$. By looking at the middle level of this cobordism, we conclude that $(X- \nu \Sigma_0) \# nS^2\times S^2$ is homeomorphic to $(X- \nu \Sigma_1) \# nS^2\times S^2$. 

It only remains to see that we can cancel off the extra $S^2\times S^2$'s. First we'll see this on the level of $\pi_2$. Since $b_2 \geq \sigma + 6$, $X$ is of the form $X' \# 3S^3\times S^2$. Therefore, since $\Sigma_0$ bounds a solid handlebody, we can assume it lies in $X'$ and therefore $X - \nu \Sigma_0$ is of the form $M \# 3S^2\times S^2$. This has equivariant intersection form of the form $\lambda \oplus 3H$ where $\lambda$ is the intersection form on $M$ and $H$ is the hyperbolic form. Therefore the equivariant intersection form on $X- \nu\Sigma_i \# nS^2\times S^2$ is $\lambda \oplus (3+n)H$. In \cite{HT}, it is explained that, as long as an intersection form over $\mathbb{Z}[\mathbb{Z}]$ has at least 3 hyperbolics, then the automorphism group of the intersection form acts transitively on hyperbolic pairs. Moreover, in \cite{CS} such automorphisms are seen to be realized by homeomorphisms (as long as there is at least one $S^2\times S^2$ summand). Therefore we can cancel off the extra $S^2\times S^2$'s. 

\end{proof}

Finally we will consider the case where $\pi_1(X - \Sigma_i)$ is finite cyclic. This case will require the most subtle conditions. The case of spheres was shown by Lee and Wilczynski and by Hambleton and Kreck with the additional assumption that $b_2(X)>|\sigma(X)|+2$. 

\begin{theorem}[\cite{LWsph}, \cite{HK}]
If X is a simply connected manifold  and $\Sigma_1$ and $\Sigma_2$ are embedded spheres in the same homology class, and moreover $\pi_1(X - \Sigma_i)$ is cyclic, then the surfaces are topologically isotopic.
\end{theorem}

We will focus on the case of genus greater than 0, but we could show something similar for spheres following our same methods. Our results will depend on the minimal genus of a surfaces. Lee and Wilczynski have shown in \cite{LWsurf} that given $x\in H_2(X)$, there is a topologically locally flat surface $\Sigma$ representing $x$ by an oriented surface of genus $g$ whose complement has cyclic fundamental group if and only if
\begin{enumerate}
\item $KS(X) = \frac{1}{8}[\sigma(X)- x^2] (mod 2)$ when $g=0$ and $x$ is characteristic, and
\item $b_2+2g \geq max_{0\leq j < d} |\sigma(X)-2j(d-j)(1/d^2)x^2|$ where $d$ is the divisibility of $x$.
\end{enumerate}

\begin{theorem}\label{t:iso3}
Let $X$ be a simply connected manifold with $b_2(X)>|\sigma(X)|+2$, and let $\Sigma_0$ and $\Sigma_1$ be embedded surfaces in the same homology class. If moreover $\pi_1(X-  \Sigma_i)$ is finite cyclic, and the surfaces have genus strictly greater than the minimum possible genus of such surfaces, then the surfaces are topologically isotopic.
\end{theorem}

\begin{proof}

Since $b_2(X)>|\sigma(X)|+2$, X is of the form $X = X' \# S^2\times S^2$ for some (topological) simply connected 4-manifold $X'$. By a theorem of Wall (\cite{Wallsph}), we can assume that $[\Sigma_i]$ is of the form $\alpha + 0 \in H_2(X') \oplus H_2(S^2\times S^2)$. By the theorem of Lee and Wilsczynski and our assumption on the genus of the $\Sigma_i$, we conclude there is a surface $S$ in $X'$ with the same genus as the $\Sigma_i$ such that $[S] = [\Sigma_i]$ in $X' \# S^2\times S^2$. Without loss of generality, we can assume $\Sigma_0 = S$.

By repeating the arguments of Theorem \ref{t:infiso} (by finding a concordance, and then analyzing the handles of the complement), we can conclude that $(X- \nu\Sigma_0) \# nS^2\times S^2$ is homeomorphic to $(X- \nu\Sigma_1) \# nS^2\times S^2$.

Now, a theorem of Hambleton and Kreck in \cite{HK} says that such stably homeomorphic manifolds are homeomorphic as long as $X- \nu\Sigma_0$ is of the form $Y \# S^2\times S^2$. But this is true by our assumption on $\Sigma_0$. Therefore $X - \nu\Sigma_0$ is homeomorphic to $X - \nu\Sigma_1$, and by extension, we have a homeomorphism of pairs $(X,\Sigma_0)\longrightarrow(X,\Sigma_1)$. Hambleton and Kreck further show \cite[Prop 4.2]{HK} that this homeomorphism can be taken to induce the identity on $H_2(X)$, and therefore, by a theorem of Quinn \cite{Quinn} the homeomorphism is isotopic to the identity.  

\end{proof}

We can also exhibit a version when $X$ is not simply connected.

\begin{theorem}
Let $\Sigma_0$ and $\Sigma_1$ be two surfaces which are locally flat embedded in a not-necessarily simply connected 4-manifold X. Suppose further that the following are satisfied:
\begin{itemize}
\item $\pi_1(\Sigma_i) \longrightarrow \pi_1(X)$ is trivial for $i=0,1$.
\item $[\Sigma_0]= [\Sigma_1]$ in $H_2(X,\mathbb{Z}[\pi_1])$.
\item The meridian to $\Sigma_i$ is null-homotopic in $X - \Sigma_i$ for $i=0,1$
\item $\pi_1(X)$ is a good group in the sense of Freedman-Quinn, \cite{FQ}.
\end{itemize}
Then there is a homeomorphism of pairs from $(X,\Sigma_1)$ to $(X,\Sigma_2)$.
\end{theorem}

\begin{proof}
Theorem \ref{t:pi1conc} shows that such surfaces must be concordant. The complement of the concordance must be an h-cobordism by applying Meyer-Vietoris to its universal cover. Now, by the additivity of Whitehead torsion, this must actually be an s-cobordism, and therefore a product cobordism.
\end{proof}

\section{0-Concordance}\label{s:0concordance}

In his thesis \cite{M}, Paul Melvin defined the concept of 0-concordance of 2-knots in $S^4$, and showed that 0-concordant knots have equivalent Gluck twists. This is what we are generalizing in Theorem \ref{t:0concsurg}. The proof can be broken in to two steps. First show that surgeries along ribbon concordant surfaces are equivalent, and second show that a 0-concordance can be split into two ribbon concordances. We define a ribbon concordance from $\Sigma_0$ to $\Sigma_1$ to be a concordance $Y^3 \subset X\times I$ such that the regular level sets of $Y$ have only critical points of degree 0 and 1.

\begin{lemma}\label{l:0concis2rib} A 0-concordance  is diffeomorphic to the composition of two ribbon concordances.
\end{lemma}

\begin{proof} A 0-concordance $\Sigma \times I = Y \hookrightarrow X \times I$ has a handle decomposition induced by the projection $p:X\times I \rightarrow I$. Denote the level sets as $\Sigma_{\{x\}} = p^{-1}(x)\cap {Y}$
Since $H(Y, \partial_-Y) = 0$, all of the handles of $\Sigma\times I$ must cancel algebraically.\footnote{Of course they cancel geometrically after handleslides, stabilizations, and isotopies, but these cannot obviously be done ambiently in $X\times I$.} We remark that the 0-concordance condition implies that the regular level sets $\Sigma_{\{x\}}$ all look like a copy of $\Sigma$ and a collection of $S^2$'s.

We will first of all show that all of the 0-handles in this decomposition of $Y$ cancel with 1-handles algebraically, and similarly all of the 2-handles cancel with 3-handles algebraically. Suppose to the contrary that there is a 1-handle that is canceled algebraically by some 2-handles attached above it. This implies that there is a particular level set $\Sigma_{\{x\}}$ where the belt sphere of the 1-handle, $b_1$ and the attaching sphere for the 2-handle $a_2$ intersect. Since these curves must intersect algebraically non-trivially, they are both homologically non-trivial in $\Sigma_{\{x\}}$. However, this implies that $\Sigma_{\{x + \epsilon\}}$ is surgery on a homologically essential loop in $\Sigma_{\{x\}}$, that is the genus of this level set changes. This contradicts the fact that we began with a 0-cobordism.

The handles can now be re-ordered such that all of the algebraically canceling 1 and 2-handles are at a lower level than all of the 3 and 4-handles. In this way we divide the concordance into two ribbon concordances.

\end{proof}

\begin{theorem}\label{t:0conciso}
If there is a ribbon concordance $Y$ from $\Sigma_0$ to $\Sigma_1$ and $X - \Sigma_0$ is simply connected, then $(X,\Sigma_0)$ is diffeomorphic to $(X,\Sigma_1)$.
\end{theorem}

\begin{proof}

The strategy will be to construct a handlebody decomposition for $X\times I - \nu Y$, and then show that all of the handles cancel. We have seen in Section \ref{notation} how to construct a dual handlebody decomposition on $X\times I -Y$ from the induced handle decomposition on $Y$. In the case of a ribbon concordance, this will have only 1- and 2-handles attached to $(X- \nu\Sigma_0) \times I$ as follows: There is no choice for how to attach the 1-handles. Notice that  the resulting 4-manifold after the 1-handles are attached,  $X_+ = \partial_+((X- \nu\Sigma_0) \times I + \sum h^1_i)$ is $X - (\Sigma_0 \cup \bigcup_i S_i)$, where the $S_i$ are a collection of unknotted $S^2$'s in $X$. Moreover, the meridians to the $S_i$ run over the 1-handles as in Figure \ref{f:meridianis1handle}. The 2-handles are attached, as explained in Section \ref{notation}, to $X_+$ along band sums of the meridians to the surfaces. After some handleslides, we can assume that each 2-handle connects a single $S^2$ with $\Sigma_0$. (See Figure \ref{f:canceling1handles}a).

	\begin{figure}
\labellist
\small\hair 2pt
\pinlabel $S_i$ at 273 55
\pinlabel $X_+$ at 414 54
\pinlabel $\Sigma_0$ at 409 88
\pinlabel $\partial_+$ at -7 65
\pinlabel $=$ at 197 61
\endlabellist	
\includegraphics{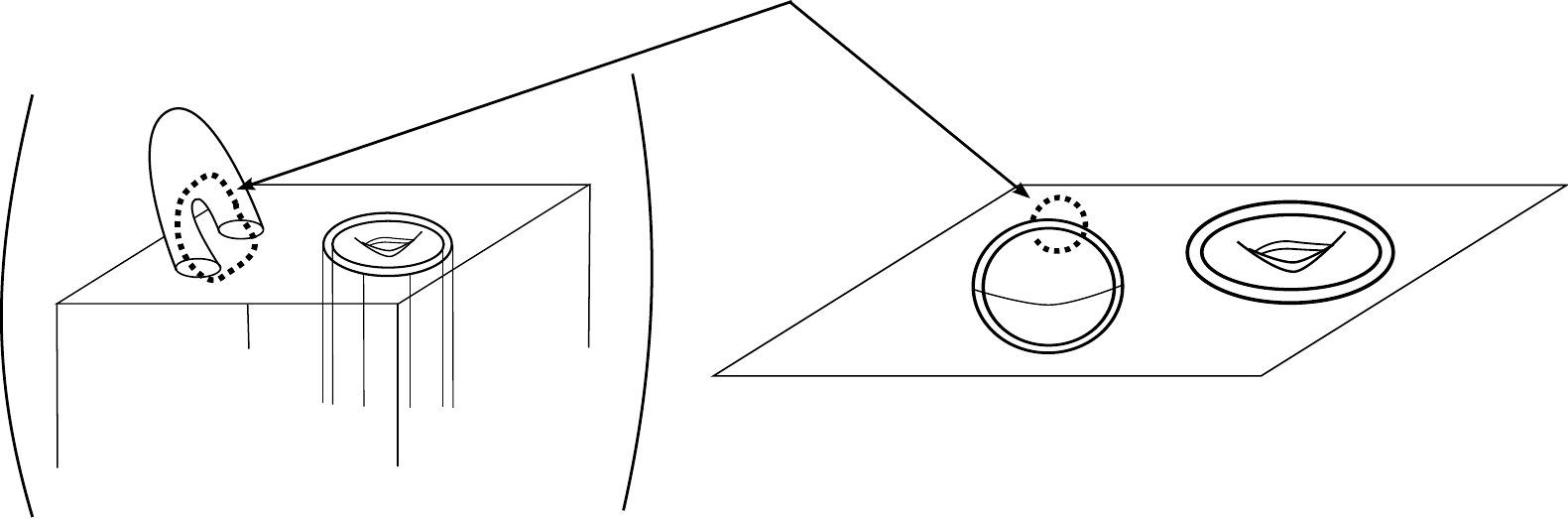}
	\caption{The meridians to the $S_i$ span the 1-handles.}
	\label{f:meridianis1handle}
	\end{figure}
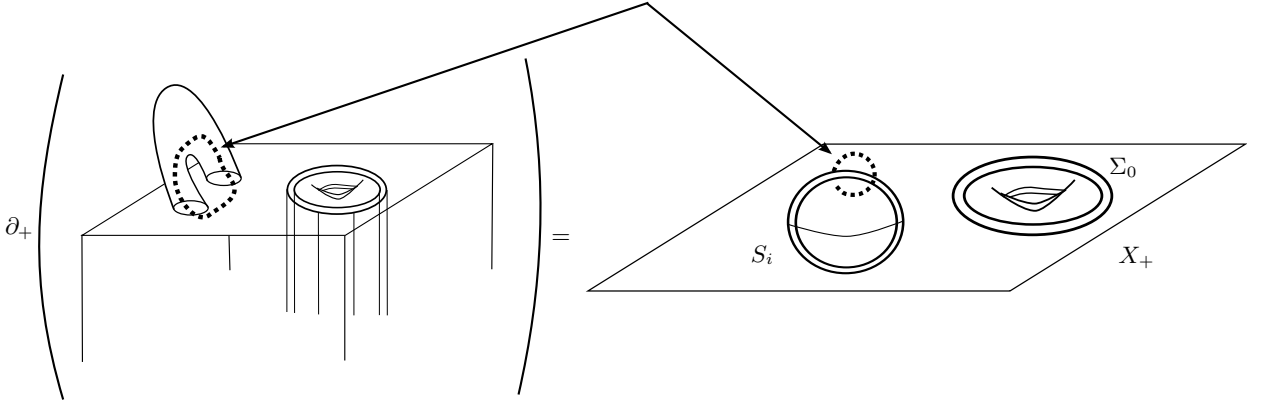 

Since the meridian to $\Sigma_0$ is null homotopic in $X - \Sigma_0$, these 2-handles can be pulled back, as in Figures \ref{f:canceling1handles}b and \ref{f:canceling1handles}c, to meridians of the $S^2$'s. But these meridians correspond to loops transversing the 1-handles (recall Figure \ref{f:meridianis1handle}). Therefore the 1-handles, cancel with the 2-handles.






	\begin{figure}
\labellist
\small\hair 2pt
\pinlabel (a) at 14 224
\pinlabel (b) at 14 120
\pinlabel (c) at 14 32
\endlabellist	
\includegraphics{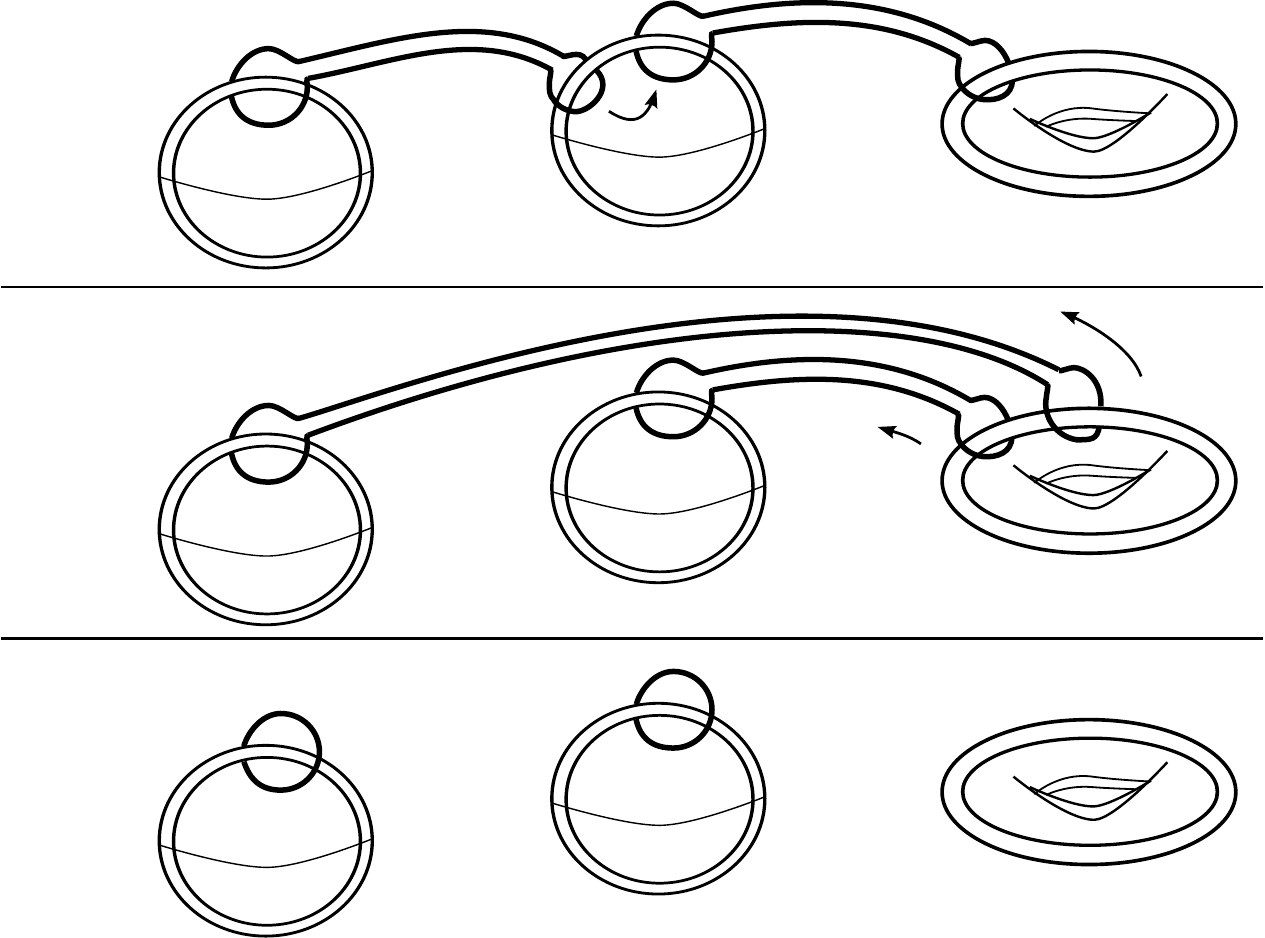}
	\caption{How the 2-handles are attached to $X_+$}
	\label{f:canceling1handles}
	\end{figure}

\end{proof}

This theorem should also be true in the case that we only assume $X - \Sigma_1$ is simply connected, but we are not quite able to prove this. If the Kervaire-Laudenbach conjecture is true, however, a handle argument shows that if $X-\Sigma_1$ is simply connected then $X- \Sigma_0$ is simply connected, so the theorem would follow as before. See \cite{Howie} for a nice survey of the Kervaire-Laudenbach conjecture. An affirmative answer to this conjecture would similarly simplify the statement of following theorem.

\begin{theorem}\label{t:0concsurg}

If (surgery$\times I$) along a 0-concordance in $X\times I$ results in a cobordism between simply connected 4-manifolds, then the cobordism is trivial, and the manifolds are diffeomorphic. For a ribbon concordance between $\Sigma_0$ and $\Sigma_1$, it is only necessary to assume that surgery on $\Sigma_0$ results in a simply connected manifold for the cobordism to be trivial. 

\end{theorem}

\begin{proof}
Let surgery along $\Sigma_0 \subset X$ be given by $X'= (X - \nu \Sigma_0) \cup_f M$ for some 4-manifold $M$ and a diffeomorphism $f:\partial M \longrightarrow \partial \nu \Sigma_0$

First consider the case of a ribbon concordance. The strategy will be to show that all of the handles in the surgered cobordism cancel.

As in the previous proof, we can construct a handlebody decomposition of $X\times I - \nu Y$ as 

\[X\times I -\nu Y = (X- \nu \Sigma_0) \times I + \sum h_1 + \sum h_2. \]

Surgery along the concordance corresponds to attaching $M\times I$ to the $(X- \nu \Sigma_0)\times I$ part via the map $f\times id$. Now the proof proceeds exactly as in the previous lemma, by performing handleslides, then pulling back the 2-handles along the bands (which is possible in this case because the meridians to $\Sigma_0$ are null homotopic in $X-\nu\Sigma_0 \cup_f M$), and then seeing that the 1- and 2-handles of the cobordism must cancel.

A general 0-concordance can be split into two ribbon concordancecs by Lemma \ref{l:0concis2rib}, so the theorem follows as in the ribbon case once we assume that surgery on both $\Sigma_0$ and $\Sigma_1$ gives simply connected manifolds.

\end{proof}

\begin{remark}Notice that while we have stated this theorem in terms of embedded surfaces, the same applies to configurations of embedded surfaces. That is, it is relatively straightforward to define 0-concordance for configurations of surfaces, and then to carry out exactly the proof above to show that 0-concordant rational-blowdowns must be equivalent.
\end{remark}

\section*{Acknowledgments}
The author wishes to thank Paul Melvin for providing a copy of \cite{M}, Danny Ruberman, Ian Hambleton, Andras Stipsicz, Rob Schneiderman, Adam Knapp, Bob Bell, Jim Howie, and Peter Teichner for help at various points, Ron Fintushel for helpful comments on an early draft of this paper, and SUNY-Stony Brook and the Max Planck Institute-Bonn where this research was carried out.

\end{document}